\documentclass[11pt,reqno]{amsart}
\usepackage{amsthm}
\usepackage{amssymb}
\usepackage{latexsym}
\usepackage{multicol}
\usepackage{verbatim,enumerate}
\usepackage{amsmath, amscd}
\usepackage[latin1]{inputenc}

\advance\textwidth by 1.2in \advance\oddsidemargin by -.6in \advance\evensidemargin by -.6in
\usepackage{tikz}
\usetikzlibrary{decorations.markings}
\tikzstyle{vertex}=[ circle, fill, draw, inner sep=0pt, minimum size=4pt,]
\tikzstyle{edge}= [thick]

\usepackage{amsmath,amsthm,amsfonts,amssymb,amscd}
\usepackage{mathrsfs}

\newtheorem{thm}{Theorem}[section]
\newtheorem{prop}[thm]{Proposition}
\newtheorem{cor}[thm]{Corollary}

\newtheorem{lemma}[thm]{Lemma}

\theoremstyle{definition}

\newtheorem{rem}[thm]{Remark}

\numberwithin{equation}{section}

\def\cal#1{\text{$\mathcal{#1}$}}
\def\ord#1^#2{#1$^{\text{#2}}$}
\def\lie#1{\mathfrak{#1}}
\def\tlie#1{\widetilde{\mathfrak{#1}}}
\def\hlie#1{\widehat{\mathfrak{#1}}}
\def\uq#1{\text{$U_q(\lie #1)$}}
\def\uqr#1^#2{\text{$U_q^{#2}(\lie #1)$}}
\def\uqt#1{\text{$U_q(\tlie #1)$}}

\def\uqhr#1^#2{\text{$U_q^{#2}(\hlie #1)$}}
\def\us#1^#2{\text{$U_{\xi}^{#2}(\lie #1)$}}
\def\ush#1^#2{\text{$U_{\xi}^{#2}(\hlie #1)$}}
\def\dus#1^#2{\text{$\dot{U}_{\xi}^{#2}(\lie #1)$}}
\def\dush#1^#2{\text{$\dot{U}_{\xi}^{#2}(\hlie #1)$}}

\def\gbr#1{{\mbox{\boldmath ${\rm #1}$}}}

\def\wt{{\rm wt}}

\def\gr{{\rm gr}}

\def\opl_#1^#2{\text{\scriptsize$\bigoplus\limits_{\text{\normalsize$#1$}}^{\text{\normalsize$#2$}}$}}
\def\otm_#1^#2{\text{\scriptsize$\bigotimes\limits_{\text{\footnotesize$#1$}}^{\text{\footnotesize$#2$}}$}}

\def\bs#1{\boldsymbol{#1}}
\def\ol#1{\overline{#1}}
\def\hcal#1{\widehat{\mathcal{#1}}}
\def\tcal#1{\widetilde{\mathcal{#1}}}

\newcommand {\Hom}{\operatorname{Hom}}
\newcommand {\ev}{\operatorname{ev}}
\newcommand{\g}{\mathfrak{g}}
\newcommand{\h}{\mathfrak{h}}
\newcommand{\n}{\mathfrak{n}}

\newcommand{\C}{\mathbb{C}}

\newcommand{\Z}{\mathbb{Z}}

\newcommand{\A}{\mathbb{A}}

\renewcommand{\thefootnote}

\begin{document}

\title[Graded Limits of tensor product of KR-modules]{Graded limits of simple tensor product of Kirillov-Reshetikhin modules for $U_q(\tlie {sl}_{n+1})$}
\author[Matheus Brito and Fernanda Pereira
]{Matheus Brito and Fernanda Pereira}
\thanks{M.B. was partially supported by FAPESP grant 2010/19458-9 and F.P. was partially supported by FAPESP grant 2009/16309-5.}

\address{Departamento de Matemática, Universidade Estadual de Campinas, Campinas - SP - Brazil, 13083-859.}
\email{ra045329@ime.unicamp.br, matheus.bb@gmail.com}

\address{Departamento de Matemática, Divisão de Ciências Fundamentais, Instituto Tecnológico de Aeronáutica, São José dos Campos - SP - Brazil, 12228-900.}
\email{fpereira@ita.br}



\begin{abstract}
We study graded limits of simple $U_q(\tlie {sl}_{n+1})$-modules which are isomorphic to tensor products of Kirillov-Reshetikhin modules associated to a fixed fundamental weight. We prove that every such module admits a graded limit which is isomorphic to the fusion product of the graded limits of its tensor factors. Moreover, using recent results of Naoi, we exhibit a set of defining relations for these graded limits.
\end{abstract}

\maketitle

\setcounter{section}{0}

\section*{Introduction}

Let $\g$ be a finite-dimensional complex simple Lie algebra, $\tlie
g = \lie g\otimes\C[t,t^{-1}]$ the corresponding loop algebra, and
$U_q(\g)$, $U_q(\tlie g)$ their Drinfeld--Jimbo quantum groups over
$\C(q)$, where $q$ is an indeterminate.
Despite the classification of the simple $U_q(\tlie g)$-modules
being known since \cite{cp:qaa} and the great development of the
theory ever since, several other basic questions about the structure
of these representations remain essentially unanswered. One of the
methods which have been used to study the structure of simple
$U_q(\tlie g)$-modules is to understand the classical limit of these
representations, i.e., their specialization at $1$ of the quantum
parameter $q$, and regard it as a representation for the current
algebra $\g[t] = \g\otimes \C[t]$. This approach
was first considered in \cite{cp:weyl}, where the authors proved
necessary and sufficient condition of the existence of the limit.
The notion of the graded limit of a $U_q(\tlie g)$-module was then further
developed in \cite{ch:fer,cm:kr}.

In \cite{cha:minr2}, Chari introduced an important class of
finite--dimensional $U_q(\tlie g)$-modules called minimal
affinizations. The Kirillov--Reshetikhin modules are the minimal affinizations of simple modules whose highest weights are multiples of a fundamental weight. In \cite{ch:fer,cm:kr}, the authors proved that the Kirillov--Reshetikhin modules admit graded limits, described a set of defining relations for them and computed their graded characters. The graded limits of
general minimal affinizations were first studied in \cite{mou:res}
where it was conjectured a set of defining relations for them. Using the theory of Demazure modules, the conjecture was
established in \cite{na:dem, na:D} for $\lie g$ of classical type and in
\cite{ln:ming2} for type $G_2$.  It was also partially established for type
$E_6$ in \cite{mope:mine6}.

It is not hard to see that if $V$ and $W$ are simple $U_q(\tlie
g)$-modules, then the classical limit of $V \otimes W$ is not
isomorphic to the tensor product of the classical limits of $V$ and $W$. In \cite{FL}, Feigin and Loktev
introduced the notion of the fusion product of graded representations
of the current algebra. It was proved in \cite{CL06, FoL, na:weyl} that a
local Weyl module for $\lie g[t]$ is isomorphic to a fusion product of fundamental local Weyl modules. On the other hand, it was known that the quantum local Weyl modules are isomorphic to tensor products of  fundamental local Weyl modules (see \cite[Section 7.4]{cm:qblock}).
Therefore, it follows that the graded limit of a tensor product of quantum local Weyl modules is isomorphic to the fusion product of the graded limits of the corresponding factors. This result motivates the following question: Is it true that if a module $V$ for $U_q(\tlie g)$ is isomorphic to a tensor product $V_1\otimes V_2$ and all three modules admit graded limits, then the graded limit of $V$ is isomorphic to the fusion product of the graded limits of $V_1 $ and $V_2$? There are a few positive answers for this question in the case that $V$ belongs to certain subclasses of simple modules. See for instance \cite[Section 4 and 5]{CV14} where $V$ is isomorphic to particular tensor products of  Kirillov--Reshetikhin modules and \cite{bcm} for modules whose prime factors belong to a special subcategory of modules considered by Hernandez and Leclerc in \cite{HL10, HL13}.

In this paper we prove that the answer to the above question is positive in the case that $\lie g$ is of type $A$ and $V$ is a tensor product of Kirillov--Reshetikhin modules associated to an arbitrary fixed fundamental weight (Theorem \ref{t:main}). Moreover, as consequence of the results of \cite{Nao15}, this $\g[t]$-module is isomorphic to a
generalized Demazure module and a $\g[t]$-module given by generators
and relations (Corollary \ref{c:gradrel}).

The paper is organized as follows. In Section 1 we give some
background information about Lie algebras and their representations. In Section 2 we briefly recall some relevant facts about the finite-dimensional representations of quantum loop algebras. In Section 3 we state and prove our main results.

\hfill

{\bf Acknowledgements:} The authors would like to express their gratitude to V. Chari and A. Moura for helpful discussions.

\section{Lie algebras}

Throughout the paper, let $\mathbb C, \mathbb Z,\mathbb Z_{\ge m}$
denote the sets of complex numbers, integers and integers bigger than or
equal to $m$, respectively. Given a ring $\gbr R$, the underlying
multiplicative group of units is denoted by $\gbr R^\times$. Given any complex Lie algebra $\lie a$ we let $U(\lie a)$ be the universal enveloping algebra of $\lie a$.

\subsection{Basics and notation} Let $\g$ be a complex simple Lie algebra of rank $n$ and $\lie h$ a Cartan subalgebra. We identify $\lie h$ and $\lie h^*$ by means of the invariant inner product $(\cdot, \cdot)$ on $\g$ normalized such that the square length of the maximal root equals $2$. Let $I=\{1,\dots,n\}$ and $R^+$ be the set of positive roots of $\g$. We denote by $\{\alpha_i\}_{i\in I}$ and $\{\omega_i\}_{i\in I}$, the sets of simple roots and fundamental weights, respectively, while $Q,P,Q^+,P^+$ the root and weight lattices with corresponding positive cones.

We fix a Chevalley basis of $\g$ consisting of $x_{\alpha}^{\pm}\in\g_{\pm\alpha}$, for each $\alpha\in R^+$, and $h_i \in \lie h$, $i\in I$.
We also define $h_\alpha \in \lie h$, $\alpha\in R^+$, by $h_\alpha = [x_{\alpha}^+,x_{\alpha}^-]$. We often simplify notation and write $x_i^{\pm}$ in place of $x_{\alpha_i}^
{\pm}$, $i\in I$.
Let $r^{\vee}$ be the maximal number of edges connecting two vertices of the Dynkin diagram of $\g$ and let also
$$d_{\alpha} = \frac{r^{\vee}}{2}(\alpha,\alpha),\quad \check{d}_{\alpha} =\frac{r^{\vee}}{d_{\alpha}}, \quad d_i = d_{\alpha_i}, \quad \alpha \in R^+, i\in I.$$
Recall that, if $C=(c_{ij})_{i,j\in I}$ is the Cartan matrix of $\g$, i.e., $c_{ij}=\alpha_j(h_{i})$, then $d_{i}c_{ij} = d_jc_{ji}$.

Define the loop algebra of $\g$ by $\tlie g=\lie g\otimes_{\mathbb
C} \mathbb C[t,t^{-1}]$ with bracket given by $[x \otimes t^r,y
\otimes t^s]=[x,y] \otimes t^{r+s}$. We identify $\g$ with the
subalgebra $\lie g\otimes 1$ of $\tlie g$, hence, we will continue
denoting its elements by $x$ instead of $x\otimes 1$. The subalgebra
$\lie g[t]=\lie g \otimes \C[t]$ of $\tlie g$ is the current algebra
associated to $\g$.

If $\lie a$ is a subalgebra of $\g$, let $\lie a[t] = \lie a\otimes \C[t]$ and its ideal $\lie
a[t]_+ = \lie a\otimes t\C[t]$. The degree grading on $\C[t]$ defines a natural $\Z_{\geq
0}$-grading on $\lie a[t]$ and thus, also on $U(\lie a[t])$. An
element of the form $(a_1\otimes t^{r_1})\cdots(a_s\otimes t^{r_s})$
has grade $r_1+\cdots + r_s$ and we denote by $U(\lie a[t])[r]$ the
subspace of grade $r$.

The affine Kac-Moody algebra $\widehat{\g}$ is the Lie algebra with underlying vector space $\tlie g \oplus \C c \oplus \C d$ equipped with the Lie bracket given by
$$ [x \otimes t^r, y \otimes t^s] =[x, y] \otimes t^{r+s} +  r\delta_{r,-s}(x,y)c,\quad [c,\hlie g]=0\quad {\rm and}\quad [d,x\otimes t^{r}] =  r x\otimes t^r,$$
for any $x,y\in \lie g$, $r,s\in \Z$. A Cartan subalgebra $\hlie h$ and a Borel subalgebra $\hlie b$ are defined as follows:
$$\hlie h=\h \oplus \C c \oplus \C d, \qquad \hlie b =\hlie h\oplus \n^+ \oplus \g \otimes t\C[t].$$
Set $\hlie n^+ =  \n^+ \oplus \g \otimes t\C[t].$

We often consider $\h^*$ as a subspace of $\widehat{\h}^*$ by
setting $\lambda(c) = \lambda(d) = 0$ for $\lambda \in \h^*$. The
root system and positive root associated to the triangular
decomposition $\hlie g = \hlie n^-\oplus \hlie h \oplus \hlie n^+$
will be denoted by $\widehat R$, $\widehat R^+$, respectively. Let
$\theta\in R^+$ be the highest root, $\delta\in \hlie h^*$ be such
that $\delta(d)=1$, $\delta(c) = \delta(h)=0$, $h\in \lie h$, and
$\alpha_0 = -\theta + \delta$. Then, if we set $\widehat I = I\sqcup
\{0\}$, we have that $\widehat\Delta = \{\alpha_i\}_{i\in \widehat
I}$ is the set of simple roots of $\hlie g$, and
$$\widehat R^+ = (R + \Z_{\geq 1} \delta) \cup R^+ \cup \Z_{\geq 1}\delta.$$

The elements
$x_\alpha^\pm\otimes t^r, x_i^\pm\otimes t^r$, and $h_i\otimes t^r$
will be denoted by $x_{\alpha,r}^\pm, x_{i,r}^\pm$, and $h_{i,r}$,
respectively. Set also, $x_0^{\pm} = x_{\theta,\pm 1}^{\mp}$. Then
$$h_0 := [x_0^+,x_0^-] = c - h_\theta.$$

Let $\widehat Q = \oplus_{i\in \widehat I}\Z\alpha_i$ and $\widehat
Q^+ = \oplus_{i\in \widehat I}\Z_{\geq 0}\alpha_i$.
Let also $\Lambda_0 \in \widehat{\h}^*$ be the unique element
satisfying $\Lambda_0(c) = 1$ and $\Lambda_0(\h) = \Lambda_0(d) =
0$. Then $\hlie h^* = \lie h \oplus  \C\delta \oplus \C\Lambda_0.$
Define $\Lambda_i \in \hlie h^*$, $i\in I$, by the requirement
$\Lambda_i(d)=0$, $\Lambda_i(h_i) = \delta_{i,j}$, $j\in \widehat I$,
and note that $\Lambda_i = \omega_i + \omega_i(h_\theta)\Lambda_0$,
for $i \in I.$ Let $\widehat P = \oplus_{i=0}^n \Z\Lambda_i \oplus
\Z\delta$ and $\widehat P^+ = \oplus_{i=0}^n\Z_{\geq 0}\Lambda_i
\oplus \Z\delta$. Equip $\hlie h^*$ with the partial order $\lambda \leq
\mu$ if and only if $\mu-\lambda\in \widehat Q^+$. Let $\hcal W$ denote the
affine Weyl group, which is generated by the simple reflections
$s_i, \ i\in \widehat I$, where
$$s_i(\mu) = \mu - \mu(h_i)\alpha_i, \quad \mu\in \hlie h^*.$$
The length of $w\in \hcal W$ will be denoted by $\ell(w)$. Recall
that the subgroup of $\hcal W$ generated by $s_i,i\in I$, is the Weyl group $\cal W$ of $\g$ and we denote its longest element by
$w_0$. Let $L = \oplus_{i\in I} \check d_{\alpha_i}\omega_i$ be the
co-weight lattice and $M = \oplus_{i\in I}\check
d_{\alpha_i}\alpha_i$ the co--root lattice. Given $\mu\in \lie h^*$,
we define $t_{\alpha}\in GL(\hlie h^*)$ by
\begin{equation}\label{e:extaffaction}
t_\mu(\lambda)= \lambda-(\lambda,\mu)\delta,\ \  \lambda\in\lie{h}^*\oplus\C\delta,\ \ \ t_{\mu}(\Lambda_0)=\Lambda_0+\mu-\frac12(\mu,\mu)\delta.
\end{equation}
Defining $T_M = \{t_{\mu} \in GL(\hlie h^*)| \mu\in M\}$ we have
$\hcal W =\cal W\ltimes T_M$. The extended affine Weyl group $\tcal
W $ is the semi--direct product $\cal W\ltimes T_L$, where $T_L =
\{t_{\mu} \in GL(\hlie h^*)| \mu\in L\}$. We also have $\tcal W =
\hcal W\ltimes\cal T$, where $\cal T$ is the group of diagram
automorphisms of $\widehat{\lie g}$. The length function $\ell$ is
extended to $\tcal W$ by setting $\ell(w\tau) = \ell(w)$, for all
$w\in \hcal W$ and $\tau \in \cal T$. The following lemma was proved
in \cite{CSVW}.
\begin{lemma}\label{l:length}
Given $\lambda, \mu \in P^+$ and $w\in \cal W$, we have
$$\ell(t_{-\lambda}t_{-\mu}w) = \ell(t_{-\lambda})+ \ell(t_{-\mu}w).$$
\end{lemma}

\subsection{Graded $\g[t]$-modules} A graded representation of $\g[t]$ is a $\Z$-graded vector space which admits a compatible Lie algebra action of $\g[t]$, i.e.,
$$V = \bigoplus_{r\in \Z}V[r], \qquad (\g\otimes t^r)V[s] \subseteq V[s+r], \ \ s\in \Z,\ \ r\in \Z_{\geq 0}.$$
A graded morphism of $\g[t]$-modules is a degree zero morphism of
$\g[t]$-graded modules. For $r\in \Z$, we let $\tau_r V$ be the $r$-th graded shift of $V$.
Given $a\in \C$, let $\ev_a:\g[t] \to \g$ be the evaluation map $x\otimes f(t)\mapsto f(a)x$. Therefore, if $W$ is a $\g$-module we can define a
$\g[t]$-module structure on $W$ by taking the pull--back
by $\ev_a$. This $\g[t]$-module is denoted by $\ev_a W$ and it is
clearly irreducible if and only if $W$ is an irreducible
$\g$-module. Moreover, $\ev_0 W$ is a graded $\g[t]$-module such
that
$$(\ev_0 W)[0] = W\quad {\rm and} \quad U(\g\otimes t\C[t])(\ev_0 W) = 0.$$

\subsection{Demazure modules}
Recall that a weight module for $\hlie h$ is one where $\hlie h$ acts diagonally. For $\Lambda \in \widehat P^+$, let $\widehat V(\Lambda)$ be the irreducible highest weight integrable $\hlie g$-module generated by an element $v_{\Lambda}$ with defining relations
$$\hlie n^+ v_{\Lambda} = 0, \quad h_i v_{\Lambda} = \Lambda(h_i) v_{\Lambda}, \quad (x_{\alpha_i}^-)^{\Lambda(h_i)+1}v_{\Lambda}=0, \quad i\in \widehat I.$$
Then
$$\widehat V(\Lambda)_{\mu} \neq \{0\} \quad \textrm{only if}\quad \mu \in \Lambda - \widehat Q^+.$$

The following proposition is well--known (see \cite[Chapters 10,11]{Kac} for instance).
\begin{prop}\label{tp}
\begin{enumerate}[(i)] \item Let $\Lambda\in \widehat P^+$.  Then $$\dim V(\Lambda)_{w\Lambda}=1,\ \ {\rm{ for \ all}}\ \ w\in\widehat W.$$
\item Given $\Lambda',\Lambda''\in\widehat P^+$, let $\Lambda = \Lambda' + \Lambda''$. Then
$$\dim\Hom_{\widehat{\lie g}}\left(V(\Lambda), V(\Lambda')\otimes V(\Lambda'')\right)=\begin{cases} 1,\ \ \Lambda=\Lambda'+\Lambda'',\\ 0,\ \ \Lambda\notin\Lambda'+\Lambda''-\widehat Q^+.\end{cases}$$
 Moreover, for all  $w\in\widehat W$, we have
$$\left (\widehat V(\Lambda')\otimes \widehat V(\Lambda'')\right)_{w\Lambda} =\widehat  V(\Lambda)_{w\Lambda},$$ where we have identified $\widehat V(\Lambda)$ with its image in $\widehat V(\Lambda')\otimes\widehat  V(\Lambda'')$.
\end{enumerate}
\end{prop}

Given $\Lambda\in \widehat P^+$ and $w\tau\in \tcal W$, with $w\in \hcal W$ and $\tau \in \cal T$, the subspace $\widehat V(\tau\Lambda)_{w\tau\Lambda}$ is one-dimensional and we fix a non-zero vector $v_{w\tau \Lambda}$ of this weight space. The Demazure module $D(w\tau\Lambda)$ is the $\hlie b$-module of $\widehat V(\tau \Lambda)$ defined by $$D(w\tau\Lambda) = U(\hlie b)v_{w\tau\Lambda}.$$

In this paper we consider the following generalization of Demazure
modules introduced in \cite{na:dem}. Given $m\in\Z_{\geq 1}$ and pairs
$(w_r,\Lambda^r)\in \tcal W \times \widehat P^+$, $1\leq r \leq m$,
set
\begin{equation*}
D(w_1\Lambda^1,\ldots, w_m\Lambda^m) = U(\hlie b)(v_{w_1\Lambda^1}\otimes\cdots \otimes v_{w_m\Lambda^m})\subseteq D(w_1\Lambda^1)\otimes\cdots\otimes D(w_m\Lambda^m).
\end{equation*}

Our primary focus in this paper are Demazure modules and its generalizations such that $w_r\Lambda^r(h_i)\leq 0$, for all $i\in I$, $1\leq r\leq m$. In this case we have $\lie n^- v_{w_r\Lambda^r} =0$, $1\leq r \leq m$, and then $D(w_1\Lambda^1,\ldots, w_m\Lambda^m)$ is a module for the parabolic subalgebra $\hlie b\oplus \lie n^-$, i.e.,
\begin{equation}\label{e:gendemgt}D(w_1\Lambda^1,\ldots, w_m\Lambda^m) = U(\g[t])(v_{w_0w_1\Lambda^1}\otimes\cdots \otimes v_{w_0w_m\Lambda^m}),
\end{equation}
c.f. \cite[Proposition 16]{Rav15}. Given $(\ell,\lambda)\in \Z_{\geq 1}\times P^+$, there exists unique $\Lambda\in \widehat P^+$ and $w\in \tcal W$ such that $$w\Lambda = w_0\lambda + \ell\Lambda_0$$ and we shall denote the $\g[t]$-module $D(w\Lambda)$ by $D(\ell,\lambda)$. In \cite{CV14}, it is given a finite presentation for the $\g[t]$-modules $D(\ell,\lambda)$ which we recall now for the simply laced case (see \cite[Theorem 2]{CV14} for complete generality).

\begin{prop}\label{p:demrel}
Assume that $\g$ is simply laced and let $(\ell,\lambda)\in \Z_{\geq 1}\times P^+$. The $\g[t]$-module $D(\ell,\lambda)$ is generated by an element $v_{\ell,\lambda}$ satisfying the following defining relation:
\begin{gather}(x_i^+\otimes 1)v_{\ell,\lambda}=0,\ \ \ (h_i\otimes t^r)v_{\ell,\lambda} = \lambda(h_i)\delta_{r,0}v_{\ell,\lambda},\ \ \ (x_i^{-})^{\lambda(h_i)+1}v_{\ell,\lambda} = 0,\label{eq1}\\
(x_{\alpha}^-\otimes t^{s_{\alpha}})v_{\ell,\lambda} = 0,\quad \alpha\in R^+,\label{eq2}\\
(x_{\alpha}^-\otimes t^{s_{\alpha}-1})^{m_{\alpha}+1}v_{\ell,\lambda} = 0,\quad \alpha\in R^+,\label{eq3}
\end{gather}
where $\lambda(h_{\alpha}) = (s_{\alpha}-1)\ell + m_{\alpha}$, with $0<m_{\alpha}\leq \ell$. Moreover, if $m_{\alpha}=\ell$, \eqref{eq3} is a consequence of \eqref{eq1} and \eqref{eq2}. In the particular case when $\ell=1$, the relations \eqref{eq2} and \eqref{eq3} follow from \eqref{eq1}.
\end{prop}
We declare the grade of $v_{\ell,\lambda}$ to be zero and, since the defining relations of $D(\ell,\lambda)$ are graded, it follows that $D(\ell,\lambda)$ is a graded $\g[t]$-module. The following is a consequence of Proposition \ref{tp} (ii).
\begin{lemma}\label{l:DequiGD}
Let $\lambda \in P^+$ and $\ell \in \Z_{\geq 1}$. Then
$$D(\ell,\ell\lambda) \cong_{\g[t]} U(\lie g[t])v_{1,\lambda}^{\otimes \ell}\subseteq D(1,\lambda)^{\otimes \ell}.$$\qed
\end{lemma}

\subsection{Fusion product}
 We recall the notion of fusion product of finite dimensional cyclic graded $\lie g[t]$--modules introduced in \cite{FL}.
Let  $V$ be  a finite--dimensional cyclic $\lie g[t]$--module
generated by an element $v$. We define a filtration $F^rV$, $r\in\mathbb Z_{\geq 0}$, on $V$ by
$$F^rV = \left(\bigoplus_{0\leq s \leq r} U(\lie g[t])[s]\right)\cdot v.$$
The associated graded vector space $\gr V$ acquires a graded $\lie g[t]$--module structure
in a natural way and is generated by the image of $v$ in $\gr V$.

Let $p\in \Z_{\geq 1}$. Let $\lambda_1,\ldots, \lambda_p$ be a sequence of elements of $P^+$ and $z_1, \ldots, z_p$ pairwise distinct complex numbers. Then
$$\mathbf V(\mathbf z): =\ev_{z_1}V(\lambda_1)\otimes \cdots \otimes \ev_{z_p}V(\lambda_p),$$
is a finite dimensional cyclic $\g[t]$-module, where $V(\lambda_s)$ is the finite--dimensional irreducible $\g$-module of highest weight $\lambda_s$, $1\leq s \leq p$. Then, the $\g[t]$-module $\gr \mathbf V(\mathbf z)$ is called the fusion product of $V(\lambda_1),\ldots, V(\lambda_p)$ and denote by $$V(\lambda_1)\ast \cdots \ast V(\lambda_p).$$
Clearly the definition of the fusion product depends on the parameters $z_s$, $1\le s\le p$. However it is conjectured in \cite{FL}  and (proved in certain cases by various people,  \cite{CL06},  \cite{FF}, \cite{FL}  \cite{FoL}, \cite{Kedem} for instance)  that the fusion product is independent of the choice of the complex numbers, hence we suppress this dependence in our notation. Note that, by definition we have
\begin{equation}\label{e:fusiontensor}
V(\lambda_1)\ast \cdots \ast V(\lambda_p)\cong_{\g}V(\lambda_1)\otimes \cdots \otimes V(\lambda_p).
\end{equation}

\section{Quantum algebras and graded limits}

\subsection{Basics and notation} We give a brief reminder on quantum loops algebras and their finite-dimensional representations. We refer the reader to \cite{cp:book} for the basic definitions.

Let $\C(q)$ be the field of rational functions in an indeterminate
$q$ and $\A = \Z[q,q^{-1}]$. Let $U_q(\lie g)$ and $U_q(\tlie g)$ be
the quantized enveloping algebras over $\C(q)$ associated to $\lie
g$ and $\tlie g$, respectively. The algebra $U_q(\g)$ is isomorphic
to a subalgebra of $U_q(\tlie g)$. Let $U_\A(\g)$ and $U_\A(\tlie
g)$ be the $\A$-form of $\uq g$ and $\uqt g$ defined in
\cite{Lus93}.  These are free subalgebras such that
$$U_q(\g) \cong U_\A(\g)\otimes_\A \C(q)\quad U_q(\tlie g) \cong U_\A(\tlie g)\otimes_\A \C(q).$$
Regarding $\C$ to be the $\A$-module by letting $q$ act as $1$, the
algebras $U_\A(\g)\otimes_\A \C$ and $U_\A(\tlie g)\otimes_\A \C$
over $\C$ have $U(\g)$ and $U(\tlie g)$ as canonical quotients. We
also recall that $U_q(\tlie g)$ is a Hopf algebra and that
$U_\A(\tlie g)$, $U_q(\g)$ and $U_\A(\g)$ are Hopf subalgebras.

It is well known that the isomorphism classes of irreducible finite-dimensional representations of $U_q(\g)$ are indexed by elements of $P^+$. Given $\lambda\in P^+$, we denote by $V_q(\lambda)$ an element of the corresponding isomorphism class. Moreover, the category of finite-dimensional $U_q(\g)$-modules is semisimple.

Let $\cal P_q^+$ be the multiplicative monoid of $n$-tuples of
polynomials $\bs\pi = (\pi_1(u),\ldots, \pi_n(u))$,
$\pi_i(u) \in \C(q)[u]$, for an indeterminate $u$, such that
$\pi_i(0)=1$ for all $i\in I$. We shall only be interested in the
submonoid $\cal P^+$ of elements $\bs\pi\in \cal P_q^+$ such that
$\pi_i(u)$ splits into linear factors in $\C(q)$, for all $i\in
I$.

Given $a\in\mathbb C(q)^\times$ and $i\in I$, let $\bs\omega_{i,a}\in \cal P^+$ be the fundamental $\ell$-weights, defined by
$$(\bs\omega_{i,a})_j(u) = 1-\delta_{i,j}au.$$
Observe that $\cal P^+$ is the free abelian monoid generated by $\{\bs\omega_{i,a}: i\in I, a\in \C(q)^{\times}\}$, and denote by $\cal P$ the corresponding free abelian group. Let also $\cal P_{\Z}^+$ be the submonoid of $\cal P^+$ of elements $\bs\pi\in\cal P^+$ where $\pi_i(u)$ has its roots in $q^{\Z}$, for all $i\in I$.

Consider the group homomorphism (weight map) $\wt:\cal P \to P$ by
setting $\wt(\bs\omega_{i,a})=\omega_i$.

It was proved in \cite{cp:qaa, CPbanff, CPnato} that the isomorphism classes of irreducible finite-dimensional representations of $\uqt g$ is indexed by $\cal P_q^+$. Given $\bs\pi \in \cal P_q^+$, we let $L_q(\bs\pi)$ be an irreducible representation in the corresponding isomorphism class. The module $L_q(\bs\pi)$ is said to be an affinization of $V_q(\lambda)$ if $\wt(\bs\pi) = \lambda$. Two simple $U_q(\tlie g)$-modules are said to be equivalent if they are isomorphic as $U_q(\g)$-modules.

It will be convenient to introduce the following notation. Given
$i\in I, a\in\mathbb C^\times, m\in\mathbb Z_{\ge 1}$, define
$$\bs\omega_{i,a,m} = \prod_{j=0}^{m-1} \bs\omega_{i,aq^{d_i(m-1-2j)}}.$$
The modules $L_q(\bs\omega_{i,a,m})$ are called Kirillov-Reshetikhin
modules. Given $\bs\pi\in\cal P^+$, there exist unique
$m_i\in\Z_{\ge 0}$, $a_{ik}\in\C(q)^\times$ and $r_{ik}\in\Z_{\ge
1}$ such that
$$\bs\pi = \prod_{i\in I}\prod_{k=1}^{m_i}
\bs\omega_{i,a_{ik},r_{ik}}$$ with
$$\frac{a_{ij}}{a_{il}}\ne
q^{\pm d_i(r_{ij}+r_{il}-2p)}\quad\text{and}\quad
\sum_{k=1}^{m_i}r_{ik}=\wt(\bs\omega)(h_i)$$ for all $i\in I$, $j\ne
l$ and $0\le p<\min\{r_{ij},r_{il}\}.$ This decomposition is called
$q$-factorization of $\bs\pi$.

In the next theorem we collect important results of $U_q(\tlie
{sl}_{n+1})$-modules. The first item is \cite[Theorem
3.5]{cp:small}, and the second item is the dual of \cite[Theorem
6.1, Corollary 6.2]{cha:braid}, for our case of interest.
\begin{thm}\label{t:sl2tensor}
Assume $\g=\lie{sl}_{n+1}$.
\begin{enumerate}
\item For all $i\in I$, $a\in \C(q)^{\times}$ and $m\in \Z_{\geq 0}$ we have $L_q(\bs\omega_{i,a,m})\cong_{\uq g} V_q(m\omega_i)$.
\item Let $m\in \Z_{\geq 1}$, $i_j\in I$, $a_j \in \C(q)^{\times}$, $n_j\in \Z_{\geq 1}$, $1\leq j \leq m$, be such that
 \begin{equation}\label{e:cyclcond}
r>s \Longrightarrow \dfrac{a_{s}}{a_{r}} \neq q^{n_s+n_r + 2 -2p +2k -i_r-i_s},
\end{equation} for all $1\leq p \leq \min\{n_s,n_r\}$ and $\min\{i_r,i_s\}<k+1\leq \min\{i_r+i_s,n+1\}$. Then $L_q(\prod_{j=1}^m\bs\omega_{i_j,a_j,n_j})$ is the unique irreducible submodule of
\begin{equation}\label{e:tensorproddec} L_q(\bs\omega_{i_1,a_1,n_1})\otimes \ldots \otimes L_q(\bs\omega_{i_m,a_m,n_m}).
\end{equation} Moreover, if \eqref{e:cyclcond} holds for all $1\leq r, s \leq m$, then the module described in \eqref{e:tensorproddec} is irreducible.
\end{enumerate}
In particular, if $\bs\pi = \prod_{j=1}^m \bs\omega_{i,a_j,n_j}$ is its $q$-factorization, then  $$L_q(\bs\pi)\cong_{U_q(\tlie g)}L_q(\bs\omega_{i,a_1,n_1})\otimes \cdots \otimes L_q(\bs\omega_{i,a_m,n_m}).$$\qed
\end{thm}

\subsection{The modules $L(\bs\pi)$, $\bs\pi\in \cal P_{\Z}^+$}
In this section we assume that $\g$ is of classical type. We recall
the definition of the $\g[t]$--modules $L(\bs\pi), \bs\pi\in \cal
P_\Z^+$. We refer the reader to \cite[Section 2.1]{bcm} and references therein for a detailed exposition.

It was shown in \cite[Section 4]{cp:weyl} that, given $\bs\pi\in \cal P_\Z^+$, the module $L_q(\bs\pi)$ admits an $\A$-form $L_\A(\bs\pi)$ and there is an action of $\tlie g$ on $\ol{L_q(\bs\pi)}:=L_\A(\bs\pi)\otimes_\A \C$. Moreover, as $\tlie g$-module, $\ol{L_q(\bs\pi)}$ is generated by a vector $v_{\ol{\bs\pi}}$ which satisfies the relations:
$$x_{i,s}^+v_{\ol{\bs\pi}}=0,\quad h_{i,r}v_{\ol{\bs\pi}}=\wt(\bs\pi)(h_i)v_{\ol{\bs\pi}},\quad (x_{i,0}^-)^{\wt(\bs\pi)(h_i)+1}v_{\ol{\bs\pi}}=0.$$
By restricting the action of $\tlie g$ to the subalgebra $\g[t]$ one can regard $\ol{L_q(\bs\pi)}$ as a module of $\g[t]$ generated by $v_{\ol{\bs\pi}}$. The $\g[t]$-module $L(\bs\pi)$ is then defined by the pullback of $\ol{L_q(\bs\pi)}$ by the $\g[t]$-automorphism $x\otimes f(t) \to x\otimes f(t-1)$.

In sum, the following is a consequence of \cite[Section 4]{cp:weyl} and the main result of \cite{CL06}.
\begin{thm}\label{c:Wprojlim}
For $\bs\pi\in \cal P_{\Z}^+$, the $\g[t]$-module $L(\bs\pi)$ is generating by an element $v_{\ol{\bs\pi}}$ satisfying
\begin{equation}\label{e:limrel}(x_i^+ \otimes\C[t])v_{\overline{\bs\pi}}=0,\ \ \ \   (h_i\otimes t^r)v_{\overline{\bs\pi}} =  \delta_{r,0}\wt(\bs\pi)(h_i)  v_{\overline{\bs\pi}},\ \ \ \ (x_i^-\otimes 1)^{\wt(\bs\pi)(h_i)+1}v_{\overline{\bs\pi}}=0.
\end{equation}
Moreover,
\begin{enumerate}
\item  $\dim L_q(\bs\pi) = \dim L(\bs\pi)$,
\item if $L_q(\bs\pi)\cong V_q(\wt (\bs\pi))$, then $L(\bs\pi)\cong_{\g[t]}\ev_0 V(\wt(\bs\pi))$, and
\item if $$L_q(\bs\pi)\cong_{\uqt g} L_q(\bs\omega_{i_1,a_1})\otimes \cdots \otimes L_q(\bs\omega_{i_p,a_p}),$$
for some $p\in \Z_{\geq 1}$ and $(i_j,a_j)\in I\times q^{\Z}$, $1\leq j \leq p$, then the relations \eqref{e:limrel} are defining relations of $L(\bs\pi)$.
\end{enumerate}
\qed
\end{thm}

The next result will be very useful in the proof of our main result (see \cite[Lemma 2.20 and proof of Proposition 3.21]{mou:res}).
\begin{lemma}\label{l:limindmap} Let $r\in \Z_{\geq 1}$. Let $\bs\pi_j\in \cal P_\Z^+$, for all $1\leq j \leq r$, and set $\bs\pi = \prod_{j=1}^r\bs\pi_j$. Assume also that
there exists a map of $U_q(\tlie g)$--modules $$L_q(\bs\pi)\to L_q(\bs\pi_1)\otimes \cdots \otimes L_q(\bs\pi_r).$$
 Then there exists a map of
$\g[t]$--modules $$L(\bs\pi)\to L(\bs\pi_1)\otimes \cdots \otimes L(\bs\pi_r),$$
 mapping $v_{\ol{\bs\pi}}\to v_{\ol{\bs\pi_1}}\otimes \cdots\otimes v_{\ol{\bs\pi_r}}$.\qed
\end{lemma}

\section{Main theorem and proof}
For the remainder of the paper let $\g$ be of type $A_n$. Let $i\in I$ and $m\in \Z_{\ge 1}$. Given $\xi =
(\xi_1\ge\xi_2\ge\ldots\ge \xi_\ell)$ a partition of $m$, define
$$\bs\pi_{i,\xi} = \prod_{j=1}^{\ell} \bs\omega_{i,q^{\xi_j-1},\xi_j}\in \cal P_\Z^+.$$
One easily checks that the above presentation of $\bs\pi_{i,\xi}$ is its $q$-factorization. In particular, by Theorem \ref{t:sl2tensor}, we have
\begin{equation}\label{e:decomppart} L_q(\bs\pi_{i,\xi})\cong_{\uqt g} L_q(\bs\omega_{i,q^{\xi_1-1},\xi_1}) \otimes \cdots \otimes L_q(\bs\omega_{i,q^{\xi_\ell-1},\xi_\ell})\cong_{\uq g} V_q(\xi_1\omega_i)\otimes \cdots \otimes V_q(\xi_{\ell}\omega_i).
\end{equation}

We state the main result of the paper.

\begin{thm}\label{t:main}

Let $i\in I$, $m \in \Z_{\geq 1}$ and $\xi = (\xi_1\geq \cdots \ge \xi_{\ell})$ be a partition of $m$. Then
$$L(\bs\pi_{i,\xi})\cong_{\g[t]} V(\xi_1\omega_i)\ast \cdots \ast V(\xi_\ell\omega_i).$$
\end{thm}

\begin{rem}By the definition of Kirillov-Reshetikhin modules and Theorem \ref{t:sl2tensor}(i), if
$$V = L_q(\bs\omega_{i,a_1,r_1})\otimes \cdots \otimes L_q(\bs\omega_{i,a_\ell, r_\ell}),$$
for some $(a_j,r_j)\in \C(q)^{\times}\times \Z_{\geq 1}$, then
$$V \cong_{U_q(\g)} V_q(r_1\omega_i)\otimes \cdots \otimes V_q(r_\ell\omega_i).$$
Therefore, setting $\xi$ to be the partition of $r_1+\cdots + r_\ell$ whose parts are $r_j$, $1\leq j \leq \ell$, \eqref{e:decomppart} implies that $L_q(\bs\pi_{i,\xi})$ is a representative of the equivalence class of $V$ which admits graded limit.
\end{rem}
The following is a straightforward consequence of Theorem
\ref{t:main} and \cite[Theorem 3.1]{Nao15}.
\begin{cor}\label{c:gradrel}
Let $\bs\pi\in \cal P_\Z^+$ satisfying the hypothesis of Theorem \ref{t:main} and set $L_j = \xi_j + \cdots +\xi_\ell$, $1\leq j\leq \ell$. Then $L(\bs\pi)$ is isomorphic to the $\g[t]$-module generated by a vector $v$ with relations
\begin{gather*}
\lie n^+[t]v=0, \ \ (h\otimes t^s)v = \delta_{s0}L_1\omega_i(h)v,\quad {\rm for}\ \ h\in \lie h, \ s\in \Z_{\ge 0}\\
x_{\alpha}^-\otimes \C[t]v =0 \ \ {\rm for}\ \alpha \in R^+ \ {\rm with} \ \omega_i(h_{\alpha})=0,\\
(x_{\alpha}^-)^{L_1+1}v = 0 \ \ {\rm for}\ \alpha \in R^+ \ {\rm with} \ \omega_i(h_{\alpha})=1,\\
(x_{\alpha}^+\otimes t)^s(x_{\alpha}^-)^{r+s}v = 0 \ \ {\rm for}\ \alpha \in R^+,\ r,s\in \Z_{\geq 1},\ {\rm with} \ \omega_i(h_{\alpha})=1,
\end{gather*} such that $r+s \geq 1 + kr +L_{k+1}, \ \textrm{for some}\ k\in \Z_{\geq 1}$. \qed
\end{cor}

Before proving Theorem \ref{t:main} we need to set up some notation. We shall also write $\xi$ as the sequence $m_1^{b_1}m_2^{b_2}\ldots m_s^{b_s}$, such that $m_j\in\{\xi_1, \ldots,\xi_{\ell}\}$, $1\leq j \leq s$, $0<m_1<m_2<\ldots <m_s$, and $b_j>0$ is the number of times that the integer $m_j$ occurs in $\xi$.

Associated to the pair $(i,\xi)$ we can also consider a generalized Demazure module, which we denote by $D_i(\xi)$, defined by
$$D_i(\xi) = D(t_{-m_1\omega_{i^*}}(b_1\Lambda_0), t_{-m_2\omega_{i^*}}(b_2\Lambda_0), \ldots, t_{-m_s\omega_{i^*}}(b_s\Lambda_0)),$$
where $i^* = n+1-i$, for all $i\in I$. We recall that $w_0\omega_i = -\omega_{i^*}$, for all $i\in I$. Using \eqref{e:extaffaction}, we have
$$t_{-m_j\omega_{i^*}}(b_j\Lambda_0)\equiv -m_jb_j\omega_{i^*} + b_j\Lambda_0 \ \mod \C\delta, \quad j=1,\ldots, s,$$
and then, by \eqref{e:gendemgt}, we have
 $$D_i(\xi) =U(\g[t])(v_1\otimes \cdots\otimes v_s) \subseteq D(b_1,b_1m_1\omega_i)\otimes \cdots \otimes D(b_s,b_sm_s\omega_i),$$
where we write for short $v_j$ instead $v_{b_j,b_jm_j\omega_i}$, for all $1\leq j \leq s$.

Let $\xi'$ denote the conjugate partition of $\xi$, i.e., $\xi' = n_1^{\ell_1}n_2^{\ell_2}\ldots n_s^{\ell_s}$ is such that
\begin{equation}\label{e:xidual}
n_j = \sum_{k=s-j+1}^s b_k \quad {\rm and}\quad \ell_j = m_{s-j+1}-m_{s-j}, \quad\textrm{for all }\ j=1,\ldots, s,
\end{equation}
where $m_{0} = 0$.

Since $n_1<\cdots <n_s$, Lemma \ref{l:length} implies that $\ell(t_{-n_j\omega_{i^*}}) =
\sum_{k=1}^{j} \ell(t_{-(n_k-n_{k-1})\omega_{i^*}})$,
for all $1\leq j \leq s$, where $n_0=0$. Therefore, the
following theorem is straightforward from \cite[Proposition
2.7]{na:dem}, \cite[Theorem 2.1]{Nao15} and \cite[Remark
3.2]{Nao15}.
\begin{thm}\label{p:giso}
Let $i\in I$, $m\in \Z_{\geq 0}$ and $\xi = (\xi_1 \geq \cdots\ge \xi_\ell)$ be a partition of $m$. Then
$$D_i(\xi')\cong_{\g[t]} V(\xi_1\omega_i)\ast \cdots \ast V(\xi_\ell\omega_i).$$\qed
\end{thm}

Note that $$\dim (L(\bs\pi_{i,\xi})) = \dim (V(\xi_1)\ast \cdots \ast V(\xi_{\ell})),$$
by \eqref{e:fusiontensor}, \eqref{e:decomppart} and Theorem \ref{c:Wprojlim}(ii). Therefore, using Theorem \ref{p:giso}, to prove Theorem \ref{t:main} it suffices to prove the following:

\begin{prop}\label{p:surjmap}
Let $i\in I$, $m\in \Z_{\geq 1}$ and $\xi$ be a partition of $m$. There exists a surjective $\g[t]$-module homomorphism
$$L(\bs\pi_{i,\xi})\twoheadrightarrow D_i(\xi').$$
\end{prop}

We devote the remainder of this section to prove Proposition
\ref{p:surjmap}. Write $\xi = m_1^{b_1}m_2^{b_2}\ldots m_s^{b_s}$,
$s\in \Z_{\geq 1}$, and its dual
$\xi'=n_1^{\ell_1}n_2^{\ell_2}\ldots n_s^{\ell_s}$. Set
\begin{equation}\label{e:pij}
\bs\pi_j = \bs\omega_{i,a_j,\ell_j}^{n_j},\quad {\rm where}\quad a_{j} = q^{\ell_j-1 + 2\sum_{k>j}\ell_k},\ j=1, \ldots, s,
\end{equation}
and observe that $$\bs\pi_{i,\xi} = \prod_{j=1}^s \bs\pi_j.$$

\begin{prop}\label{p:qinjmap}
Let $\bs\pi_j$, $1\leq j\leq s$, as in \eqref{e:pij}. Then $L_q(\bs\pi_{i,\xi})$ is the unique irreducible submodule of $L_q(\bs\pi_s)\otimes \cdots \otimes L_q(\bs\pi_2)\otimes L_q(\bs\pi_1)$.
\end{prop}

\begin{proof}
By Theorem \ref{t:sl2tensor}(ii), it suffices to show that
\begin{equation}\label{e:compcycl}\ell_k-1+ 2\sum_{t>k}\ell_t -(\ell_j-1+ 2\sum_{t>j}\ell_t) \neq \ell_j + \ell_k +2 -2p +2g -2i,
\end{equation} for all $1\leq j< k\leq s$, $1\leq p\leq \min\{\ell_j,\ell_k\}$, $i<g+1\leq \min\{2i,n+1\}$.
This is clear, since the left hand side of \eqref{e:compcycl} is a negative integer and the right hand side of \eqref{e:compcycl} is always a non-negative integer.
\end{proof}

\begin{proof}[Proof of Proposition \ref{p:surjmap}]
By Proposition \ref{p:qinjmap} and Lemma \ref{l:limindmap}, there exists a map
\begin{equation}\label{e:mapinlim}L(\bs\pi_{i,\xi})\to \bigotimes_{j=1}^s L(\bs\pi_{j}),
\end{equation}
mapping $v_{\overline{\bs\pi}_{i,\xi}}$ to $v_{\overline{\bs\pi}_{s}}\otimes \cdots \otimes v_{\overline{\bs\pi}_{1}}$. We claim that
$$L(\bs\pi_j)\cong_{\g[t]}D(\ell_j,\ell_jn_j\omega_i), \quad \textrm{for all} \  1\leq j \leq s.$$
Assuming the claim, by \eqref{e:mapinlim}, we have a $\g[t]$-module homomorphism
$$L(\bs\pi_{i,\xi}) \to D(\ell_1, \ell_1n_1\omega_i) \otimes \cdots \otimes D(\ell_s, \ell_sn_s\omega_i),$$
whose image is $D_i(\xi')$, as required.

For the claim, let $1\leq j \leq s$ and set
$$W = L_q(\bs\varpi_{\ell_j-1})\otimes\cdots\otimes L_q(\bs\varpi_{0}), \quad {\rm where}\quad \bs\varpi_{k} = (\bs\omega_{i,a_jq^{\ell_j-1-2k}})^{n_j},\ \ \ 0\leq k \leq \ell_j-1.$$
Observe that $\bs\pi_j = \prod_{k=0}^{\ell_j-1}\bs\varpi_k$ and, by
Theorem \ref{t:sl2tensor}(ii), $$L_q(\bs\varpi_k)\cong
L_q(\bs\omega_{i,a_jq^{\ell_j-1-2k}})^{\otimes n_j}, \ \ \textrm{for
all}\ \ 0\leq j \leq \ell_j-1.$$ Therefore, by Theorem \ref{c:Wprojlim} and
Proposition \ref{p:demrel}, we have
$$L(\bs\varpi_k)\cong_{\g[t]}D(1,n_j\omega_{i}), \quad \textrm{for all} \ 0\leq k \leq \ell_j-1.$$
Moreover, arguing as in the proof of Proposition \ref{p:qinjmap}, we obtain that $L_q(\bs\pi_j)$ is the unique irreducible submodule of $W$ and, hence, there exists a $\g[t]$-module homomorphism
$$L(\bs\pi_j)\to D(1,n_j\omega_{i})^{\otimes\ell_j},$$
whose image is $D(\ell_j, \ell_jn_j\omega_{i})$, by Lemma \ref{l:DequiGD}. To conclude that such surjective homomorphism is also injective,
it suffices to show that
\begin{equation}\label{e:dimeq}\dim L(\bs\pi_j) =\dim D(\ell_j,\ell_jn_j\omega_i).
\end{equation}
Setting the partition $\psi = \ell_j^{n_j}$, we have $\psi' = n_j^{\ell_j}$ and, by Theorem \ref{p:giso}, it follows that
$$D_i(\psi') = D(\ell_j,\ell_jn_j\omega_i) \cong V(\ell_j\omega_i) \ast \cdots \ast V(\ell_j\omega_i).$$
On the other hand, by Theorem \ref{t:sl2tensor},
$$L_q(\bs\pi_j)\cong_{U_q(\tlie g)} L_q(\bs\omega_{i,a_j,\ell_j})^{\otimes n_j}\cong_{U_q(\lie g)} V_q(\ell_j\omega_i)^{\otimes n_j}.$$
By Theorem \ref{c:Wprojlim}(i) and (ii), and using \eqref{e:fusiontensor} we conclude that \eqref{e:dimeq} holds, which finishes the proof.
\end{proof}

\begin{rem}
Assume $\g = \lie{sl}_2$ and let $I=\{1\}$. It is well known that if $\bs\pi = \prod_{j=1}^m\bs\omega_{1,a_j,r_j}$ is the $q$-factorization of $\bs\pi$, then
$$L_q(\bs\pi)\cong_{U_q(\tlie g)} \otimes_{j=1}^m L_q(\bs\omega_{1,a_j,r_j}).$$
In particular, given $m\in \Z_{\geq 0}$, each affinization of
$V_q(m\omega_1)$ must be of this form (see \cite[Lemma 6.5]{CH}).
Therefore, the equivalence classes of affinizations of
$V_q(m\omega_1)$ are in bijection with the set of all partitions
$\xi$ of $m$, and each of these classes has $L_q(\bs\pi_{1,\xi})$ as
its representative. In particular, Theorem \ref{t:main} implies that
the graded limit of each class of affinization of $V_q(m\omega_1)$
is isomorphic to a fusion product.

It is also known that this statement does not hold in general. For
instance, if $\g=\lie{sl}_4$ it can easily be proved that we have
three different classes of affinizations of $V_q(2\omega_2)$, with
representatives $L_q(\bs\omega_{2,q,2})$,
$L_q(\bs\omega_{2,1}\bs\omega_{2,q^4})$ and
$L_q(\bs\omega_{2,1}\bs\omega_{2,q^6})$, for example. By Theorems \ref{t:main} and \ref{p:giso}, we have
$$L(\bs\omega_{2,q,2}) \cong_{\g[t]} V(2\omega_2) \cong_{\g[t]} D(2,2\omega_1)\ \ {\rm and}\ \  L(\bs\omega_{2,1}\bs\omega_{2,q^6})\cong_{\g[t]} V(\omega_2)\ast V(\omega_2)\cong_{\g[t]} D(1,2\omega_2).$$
Using the results of \cite{NT} we have that
$L_q(\bs\omega_{2,1}\bs\omega_{2,q^4})$ is a module corresponding to
a skew Young diagram and
$$L_q(\bs\omega_{2,1}\bs\omega_{2,q^4})\cong_{U_q(\g)} V_q(2\omega_2)\oplus V_q(\omega_1 + \omega_3).$$
Moreover, since $D(2,2\omega_2)\cong_{\g[t]} U(\g[t])(v_{1,\omega_2}\otimes v_{1,\omega_2}) \subseteq D(1,\omega_2)\otimes D(1,\omega_2)$, by Lemma \ref{l:DequiGD}, we cannot have $L(\bs\omega_{2,1}\bs\omega_{2,q^4})$ being isomorphic to a fusion product and nor a generalized Demazure module, but rather a proper quotient of $D(1,2\omega_2)$.
\end{rem}

\bibliographystyle{amsplain}

\begin{thebibliography}{10}

\bibitem{bcm}
M.~Brito, V.~Chari and A.~Moura {\it Demazure modules of level two
and prime representations of  quantum affine $\lie{sl}_{n+1}$},
arXiv:1504.00178.

\bibitem{cha:minr2}
V.~Chari, {\it Minimal affinizations of representations of quantum
groups: the rank-2 case}, Publ. Res. Inst. Math. Sci. {\bf 31}
(1995), 873--911.

\bibitem{ch:fer}
\bysame, {\it On the fermionic formula and the Kirillov-Reshetikhin
conjecture}, Int. Math. Res. Notices {\bf 12} (2001), 629--654.

\bibitem{cha:braid}
\bysame, {\em Braid group actions and tensor products}, Int. Math.
Res. Notices (2002), 357--382.

\bibitem{CH}
V. Chari and D. Hernandez, {\it Beyond Kirillov--Reshetikhin
modules, in Quantum Affine Algebras, Extended Affine Lie Algebras,
and their Applications}, Contemp. Math. {\bf 506} (2010), 49--81.

\bibitem{CL06}
V. Chari and S. Loktev, {\it Weyl, Demazure and fusion modules for
the current algebra of $\lie{sl}_{r+1}$}, Adv. Math. {\bf 207}
(2006), 928--960.

\bibitem{cm:qblock}
V.~Chari and A.~Moura, {\em Characters and blocks for
finite-dimensional representations of quantum affine algebras}, Int.
Math. Res. Notices {\bf 5} (2005), 257--298.

\bibitem{cm:kr}
\bysame, {\em The restricted Kirillov-Reshetikhin modules for the
current and twisted current algebras}, Comm. Math. Phys. {\bf 266}
(2006), 431--454.

\bibitem{cp:qaa}
V.~Chari and A.~Pressley, {\it Quantum affine algebras}, Comm. Math.
Phys. {\bf 142} (1991), 261--283.

\bibitem{cp:small}
\bysame, {\em Small representations of quantum affine algebras},
Lett. Math. Phys. {\bf 30} (1994), 131--145.

\bibitem{cp:book}
\bysame, A guide to quantum groups, Cambridge University Press
(1994).

\bibitem{CPbanff}
\bysame, \emph{Quantum affine algebras and their representations}.
Representations of groups (Banff, AB, 1994), CMS Conf. Proc. {\bf
16} (1995), 59--78.

\bibitem{CPnato}
\bysame, \emph{Quantum affine algebras and integrable quantum
systems}. Quantum fields and quantum space time (1996), 245--263,
NATO Adv. Sci. Inst. Ser. B Phys. {\bf 364} (1997).

\bibitem{cp:weyl}
\bysame, {\em Weyl modules for classical and quantum affine
algebras}, Represent. Theory {\bf 5} (2001), 191--223.

\bibitem{CSVW}
V.~Chari, P.~Shereen, R.~Venkatesh and J.~Wand, \emph{A Steinberg
type decomposition theorem for higher level Demazure modules},
arXiv:1408.4090.

\bibitem{CV14}
V. Chari and R. Venkatesh, {\it Demazure modules, fusion products
and $Q$-systems}, Comm. Math. Phys. {\bf 333} (2015), no. 2,
799--830.

\bibitem{FF}
 B.~L. Feigin and E.~Feigin,
{\em $q$-characters of the tensor products in $\lie{sl}_2$-case},
Mosc. Math. J.  {\bf 2}  (2002),  no. 3, 567--588.

\bibitem{FL} B. Feigin and S. Loktev, {\em On Generalized Kostka
Polynomials and the Quantum Verlinde Rule}, Differential topology,
infinite--di\-men\-si\-onal Lie algebras, and applications, Amer.
Math. Soc. Transl. Ser. 2, Vol. {\bf 194} (1999), p. 61--79.

\bibitem{FoL}
G.~Fourier, P.~Littelmann. \emph{Weyl modules, Demazure modules,
KR-modules, crystals, fusion products, and limit constructions},
Adv. Math. \textbf{211} (2007), no.2, 566--593.

\bibitem{HL10}D.~Hernandez and B.~Leclerc, {\it Cluster algebras and quantum affine algebras}, Duke Math.
J. {\bf 154} (2010), 265--341.

\bibitem{HL13}
\bysame, {\it Monoidal categorifications of cluster algebras of type
$A$ and $D$}, Symmetries, Integrable Systems and Representations,
Springer Proceedings in Mathematics \& Statistics {\bf 40} (2013),
175--193.

\bibitem{Kac} V.~Kac. Infinite Dimensional Lie Algebras, Cambridge University Press (1983).

\bibitem{Kedem} R.~Kedem, {\it A pentagon of identities, graded tensor products, and the Kirillov--Reshetikhin conjecture}, New
trends in quantum integrable systems, World Sci. Publ. (2011),
173--193.

\bibitem{ln:ming2}
J.R.~Li and K. Naoi, {\it Graded limits of minimal affinizations
over the quantum loop algebra of type $G_2$}, arXiv:1503.02178.

\bibitem{Lus93} G. Lusztig. Introduction to Quantum groups, Progress in Mathematics {\bf 110}, Birkhäuser
Verlag, Boston, (1993).

\bibitem{mou:res} A.~Moura, {\em Restricted limits of minimal affinizations}, Pacific J. Math. {\bf 244} (2010), 359--397.

\bibitem{mope:mine6}
A. Moura and F. Pereira, {\em Graded limits of minimal affinizations
and beyond: the multiplicity free case for type $E_6$}, Algebra and
Discrete Mathematics {\bf 12} (2011), 69--115.

\bibitem{na:weyl} K.~Naoi, \emph{Weyl Modules, Demazure modules and finite crystals for non-simply laced type}, Adv. Math. \textbf{229} (2012), no.2, 875--934.

\bibitem{Nao12}
\bysame, {\it Fusion products of Kirillov-Reshetikhin modules and
the $X = M$ conjecture}, Adv. Math. {\bf 231} (2012), 1546--1571.

\bibitem{na:dem}
\bysame, {\it Demazure modules and graded limits of minimal
affinizations}, Represent. Theory {\bf 17} (2013), 524--556.

\bibitem{na:D}
\bysame, {\it Graded limits of minimal affinizations in type $D$},
SIGMA {\bf 10} (2014), 047, 20 pages.

\bibitem{Nao15}
\bysame, {\it Defining relations of fusion products and Schur
positivity}, arXiv:1504.00109.

\bibitem{NT}
M. Nazarov and V. Tarasov, {\it Representations of Yangians with
Gelfand-Zetlin bases}, J. Reine Angew. Math. {\bf 496} (1998),
181--212.

\bibitem{Rav15}
B.~Ravinder, {\it Generalized Demazure modules and fusion products},
arXiv:1504.01537.

\end{thebibliography}

\end{document}